\documentclass[12pt,a4paper,oneside,reqno]{amsart}

\usepackage{amssymb, amstext, amscd, amsmath, amsfonts, amsthm, amscd, color}
 
\usepackage[mathscr]{euscript}  
\usepackage{tikz,mathdots,enumerate}
\usepackage{graphicx}
\usepackage{tikz-cd}

\usepackage{pb-diagram}

\usepackage{pgfplots}
\usepackage{url}

\newtheorem{thm}{Theorem}[section]

\newtheorem{cor}[thm]{Corollary}

\newtheorem{lemma}[thm]{Lemma}

\numberwithin{equation}{section}

\theoremstyle{definition}
\newtheorem{rem}[thm]{Remark}
\newtheorem{example}[thm]{Example}
\newtheorem{definition}[thm]{Definition}
\newcommand{\bD}{{\mathbb{D}}}
\newcommand{\bN}{{\mathbb{N}}}

\newcommand{\bR}{{\mathbb{R}}}

\newcommand{\bS}{{\mathbb{S}}}
\newcommand{\bP}{{\mathbb{P}}}

  \newcommand{\B}{{\mathcal{B}}}
  
  \newcommand{\D}{{\mathcal{D}}}

\renewcommand{\S}{{\mathcal{S}}}
  
  \newcommand{\U}{{\mathcal{U}}}

\usepackage{verbatim}

\usepackage{tikz} 
\usetikzlibrary{calc}

\begin{document}

\setcounter{tocdepth}{1}

\title[Noncompact surfaces, triangulations and rigidity]
{Noncompact surfaces, triangulations and rigidity}

\author[Stephen C. Power]{Stephen C. Power}


\address{Dept.\ Math.\ Stats.\\ Lancaster University\\
Lancaster LA1 4YF \\U.K. }

\email{s.power@lancaster.ac.uk}

\thanks{
{MSC2020 {\it  Mathematics Subject Classification.}
52C25, 05C10 \\
{  \today} 
}}

\maketitle

\begin{abstract}
Every noncompact surface is shown to have a (3,6)-tight triangulation, and applications are given to the generic rigidity of countable bar-joint frameworks in $\bR^3$. In particular, every noncompact surface has a (3,6)-tight triangulation that is minimally 3-rigid. 
A simplification of Richards' proof of Ker\'ekj\'art\'o's  classification of noncompact surfaces is also given.

\end{abstract}

\section{Introduction} 
Perhaps the first example of a noncompact surface that springs to mind is a sphere, torus or Klein bottle, with several points or closed discs removed.
Further examples are given by unbounded surfaces, such as an infinite cylinder with infinitely many handles, and by more exotic variations  obtained by deleting a closed totally disconnected subset $C$. 
We show that every noncompact surface has a triangulation whose underlying graph is a countable union of finite $(3,6)$-tight graphs, in the sense of Definition \ref{d:36tight}. This is first done in a constructive way, with moves that preserve minimal 3-rigidity as well as $(3,6)$-tightness.
A second nonconstructive proof follows from a characterisation of general $(3,6)$-tight triangulations of compact surfaces in terms of length constraints on the boundary walks of superfaces of a given genus.  

For compact surfaces the Euler formulae imply that only the sphere has a $(3,6)$-tight triangulation; the $(3,6)$-tight cellular embedded graphs $G\subset S$ of other compact surfaces necessarily have a number of nontriangular faces. If the boundary walks of these faces happen to be disjoint cycles of $G$ then $G$ may be regarded as a $(3,6)$-tight triangulation of the compact bordered surface formed by the excision of these faces. 

A motivation for determining $(3,6)$-tight triangulations of surfaces comes from their relevance to the rigidity or otherwise of triangulated bar-joint frameworks in $\bR^3$. A finite connected bar-joint framework $(G,p)$ in $\bR^3$, with the joints $p(v)$ located generically, is \emph{minimally rigid} 
 if it is rigid (infinitesimally or continuously) and if removing any edge results in a framework that is flexible. If this is the case we say, without ambiguity, that the graph is \emph{minimally $3$-rigid}. A necessary  condition for this, although not a sufficient one, is that $G$ is a $(3,6)$-tight graph.
The origins of such graph rigidity can be traced back to Cauchy's proof in 1813 that a convex polyhedron in $\bR^3$ is a continuously rigid plate-and-hinge structure \cite{cau}, and to Maxwell's 1864 observation that the graph of a rigid bar-joint framework satisfies certain counting rules \cite{max}. 
 
The rigidity definitions apply also to countable graphs and their infinite bar-joint frameworks in $\bR^3$ \cite{owe-pow} and to rigidity with respect to nonEuclidean norms \cite{kit-pow}, \cite{dew-kit-nix}. As a main application we show in Section \ref{s:barjointframeworks} that every noncompact surface has triangulations whose graphs are minimally 3-rigid. For $G\subset S\backslash C$, with $S$ the sphere or the projective plane and $C$ closed and totally disconnected, $(3,6)$-tightness is in fact equivalent to minimal 3-rigidity but these low genus examples are exceptions and we indicate some related open problems in Section \ref{ss:openprobs}.
 
In the construction of triangulations we first make use of certain model noncompact surfaces $S_\gamma$  and the fact that every noncompact surface (without boundary) is homeomorphic to an $S_\gamma$. Such a model surface is a topological connected sum of low genus compact surfaces where the structure graph of the connections is a countable tree. These low genus surfaces are the sphere $\bS_0$, the torus  $\bS_1$, and the projective plane $\bP$.
The model surfaces may also be described as the join of compact bordered surfaces 
of the form $S\backslash r\bD$, where $r=1,2$ or $3$, and $S$ is equal to $\bS_0, \bS_1$ or $\bP$.
Here, $r\bD$ denotes the union of the interiors of $r$ disjoint embedded closed discs.


The development is organised as follows. In Section \ref{s:modelsurfaces} we define noncompact surfaces and  model surfaces and give a sketch proof of Theorem \ref{t:modelsufficiency}, showing that every noncompact surface $S$ is homeomorphic to a model surface. The proof requires the construction of a certain inclusion chain of compact bordered subsurfaces of $S$ and corresponds to the 
\emph{canonical exhaustion} construction in Section 29 of Ahlfors and Sario \cite{ahl-sar}.
In Section \ref{s:invariants} we define homeomorphism invariants, namely the totally disconnected Hausdorff space $\beta(S)$, known as the \emph{ideal boundary} of $S$, and two of its subsets.
These invariants are the ingredients of Ker\'ekj\'art\'o's  1923 classification \cite{ker} of noncompact surfaces, Theorem \ref{t:Kthm}, the proof of which has been given in Richards \cite{ric}.  We give a proof broadly similar to this although more explicit and economical through the use of Theorem \ref{t:modelsufficiency} and model surfaces. We note that Goldman \cite{gol} has given an algebraic proof involving the calculation of singular homology groups.

In Section \ref{s:triangulations of noncompact surfaces} we show that every noncompact surface has a $(3,6)$-tight triangulation by means of an explicit construction sequence for each model surface. This involves vertex-splitting moves, Henneberg 0-extension moves and joins with appropriately sparse triangulations of low genus compact bordered surfaces. 
In Section \ref{s:characterisation} an alternative construction  of such triangulations is given in terms of barycentric subdivision and the satisfaction of length constraints on the boundary walks of superfaces of a given genus.
In the final section we give applications to the generic rigidity of infinitely triangulated bar-joint frameworks. 

For a broad background on compact surfaces, embedded graphs and triangulations see Gross and Tucker \cite{gro-tuc} and Mohar and Thomassen \cite{moh-tho}. Graver et al \cite{gra-ser-ser} give the generic rigidity theory of finite bar-joint frameworks. Generic rigidity for infinite frameworks is considered, for example, in \cite{kit-pow-I}, \cite{kit-pow-II} and \cite{kas-kit-pow}.
 
\section{The model surfaces $S_\gamma$}\label{s:modelsurfaces}

A \emph{surface} is a pathwise connected metrizable Hausdorff space for which every point $x$ has a neighbourhood that is homeomorphic to the open disc $\bD$ with the image of $x$ an interior point.
Examples of compact surfaces are obtained by gluing in a pairwise manner the sides of several polygons whose total side count is even. In this case the (paired) edges provide an embedded graph in the compact surface, the faces of  which are \emph{cellular}, that is, homeomorphic to open discs. Since the polygons may be partitioned into triangles such a surface is said to be triangulable, or to be a triangulated surface. In fact every compact surface is homeomorphic to a triangulated surface \cite{tho} and it  follows readily from this that every noncompact surface admits a triangulation. See \cite{tho} and \cite{moh-tho} for example.
 
A \emph{subsurface} of a surface $S$ (either compact or noncompact) is a pathwise connected set in $S$ that is closed in $S$ with boundary consisting  of a \emph{finite} number of disjoint simple closed curves. We usually refer to a subsurface of $S$ as a \emph{bordered surface} or \emph{bordered subsurface}. The \emph{genus} $g=g(S)$ of $S$ is defined to be either the maximum of the genus of a compact bordered subsurface of $S$ or to be infinity if there is no finite maximum. We recall that the reduced genus $g_r(A)$ of a compact bordered surface $A$ is equal to $g(A)$ if $A$ is orientable and to $g(A)/2$ otherwise.  
 
The classification theorem for compact surfaces asserts that each such  surface is homeomorphic to either a sphere with $g\geq 0$ added handles, or to a sphere with $g\geq 1$ added crosscaps \cite{gro-tuc}, \cite{moh-tho}.
We now consider a family of model noncompact surfaces that are defined in terms of possibly countable additions of handles and crosscaps in a structured manner.

\subsection{The surfaces $S_\gamma$} 
Let $\bS_0, \bS_1$ be the sphere and the torus respectively, let $\bP$ be the real projective plane, and denote the connected sum surface of a pair $S_a, S_b$ of these spaces as $S_a + S_b$. This is obtained by excising the interior of embedded closed discs $D_a\subset S_a, D_b\subset S_b$ and gluing the results together at their boundary curves. 
Multiple connected sums, such as $S_a+S_b+S_c+S_d= ((S_a+S_b)+S_c)+S_d$, may be defined in this way where, prior to gluing, the first and last component bordered surfaces have a single boundary curve, and the intermediate bordered surfaces have boundaries consisting of a pair of disjoint simple closed curves.

The infinite connected sum $S=\sum_{k=1}^\infty S_k$, where  $S_k\in \{\bS_0, \bS_1, \bP\}$ for each $k$, may now be defined as the noncompact surface given as the union of the compact bordered surfaces
\[
S^n= (S_1+\dots +S_{n-1}) +(S_n\backslash D_{n}^0), \quad n=2,3,\dots ,
\]
where $D_n^0$ is the interior of an embedded closed disc $D_n \subset S_n$ and the boundary $\partial S^n$ of $S^n$ is a single closed curve equal to the boundary $\partial D_n$. In the case that $S_k=\bS_0$ for all $k$ this infinite connected sum, $\sum_\bN \bS_0$, is homeomorphic to the plane.

In the construction of $S$ we have   \emph{component bordered surfaces}, prior to any gluing, namely  $S_1\backslash \bD$ and $S_k\backslash 2\bD$, for $k\geq 2$. 
We now define the more general infinite connected sums $S_\gamma$, where a component bordered surface can also be the \emph{branching surface} $\bS\backslash 3\bD$  (see Figure \ref{f:pants2legs}). In this case the structure of the connections is determined by a countable tree that is a  subtree of the binary tree $T_{\rm bin}$.

Let $T_{\rm bin}=(V_{\rm bin},E_{\rm bin})$ be the binary tree associated with the usual construction of the middle thirds Cantor set $C$ in $[0,1]$. A natural labelling of the vertices is by the symbols 
\[
s= \phi, 0,2,00,02,20,22,000, \dots , 222, 0000\dots 
\]
where the digit symbols $s$ of length $|s|=k$ label the closed intervals $I_s$ of length $3^{-k}$ in the construction. The vertex $v_s$ of the tree at level $k$ is adjacent to $v_{s0}$ and $v_{s2}$ at level $k+1$ and we may view the edges $v_sv_{s0}$ and $v_sv_{s2}$ as directed edges.
 Since $C$ consists of the points $x$ in $[0,1]$ with a ternary expansion
$0.s_1s_2\dots $, with $s_k=0$ or $2$, there is a bijection between $C$ and the infinite directed paths with source vertex $v_\phi$ associated with such ternary expansions.

Suppose that $T=(V,E)$ is a countable subtree  of $T_{\rm bin}$ with  $v_\phi\in V$ and that $\gamma$ is a map $\gamma:V \to \{\bS_0, \bS_1, \bP\}$ with $\gamma(v)=\bS_0$ if $v=v_\phi$ or if $\deg(v)=3$.
The noncompact surface $S_\gamma$ is the associated infinite connected sum, denoted 
$
\sum_{v\in V}\gamma(v),
$ 
whose component surfaces are $\bS_0\backslash \bD$ or $\bS_0\backslash 2\bD$ if $v=v_\phi$,
$\gamma(v)\backslash 3\bD=\bS_0\backslash 3\bD$ if $\deg(v)=3$, and $\gamma(v)\backslash 2\bD$  
if $\deg(v)=2.$

Once again, for $n\geq 0$, the surface $S_\gamma$ has a compact bordered subsurface $S_\gamma^n$ associated with the finite subtree $T_n$ of $T$ induced by the set of vertices $v_s$ of $T$ with $|s|\leq n$, and $S_\gamma$ is the union of this inclusion chain of bordered surfaces. Note also that the complement of $S_\gamma\backslash S_\gamma^n$ has closure equal to a union of disjoint bordered surfaces, each with a single cycle as boundary. 

The subset $C_T\subset C$ is defined to be the subset  corresponding to the infinite directed paths $\pi$ of $T$ that are also paths of $T_{\rm bin}$. This is a compact totally disconnected Hausdorff space in the relative topology. Also it is routine to show that it is homeomorphic to the \emph{ideal boundary} of $S_\gamma$, as defined in Section \ref{s:invariants}.

\begin{example}\label{e:Sgammaexamples}
(i) Let $T$, as a directed tree with source $v_\phi$, have no terminal vertices and let $\gamma(v)=\bS_0$ for all vertices $v$. If there are only finitely many infinite directed paths $\pi$, with source $v_\phi$,  then $C_T$ is a finite set and $S_\gamma$ is homeomorphic to the sphere with finitely many points deleted. In this case triangulations $G\subset S$ and the generic rigidity and flexibility of their bar-joint frameworks were considered in Kitson and Power \cite{kit-pow-II}. 

(ii) If $T=T_{\rm bin}$ and $\gamma(v)=\bS_0$ for every vertex then
$S_\gamma$ is homeomorphic to $\bS_0\backslash X$ where $X$ is homeomorphic to a Cantor set. This can be seen through a direct argument or as a simple case of the proof scheme in Theorem \ref{t:modelsufficiency}.

(iii) Let $T\subset T_{\rm bin}$ have the structure indicated in Figure \ref{f:treeExampleThird_C}. 
Then $C_T$ is equal to $\{z_1, z_2, \dots \}\cup \{z_\infty\}$ where $z_\infty$ is the single accumulation point and corresponds to the rightmost infinite path $\pi$. 
 \begin{center}
\begin{figure}[ht]
\centering
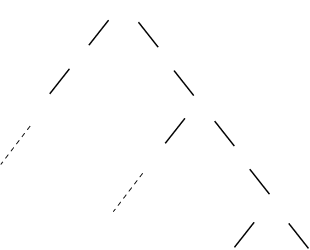\quad \quad \quad 
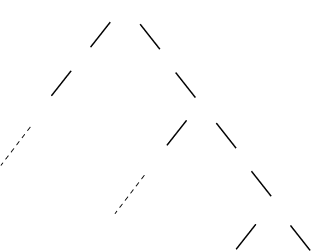
\caption{(i) A subtree $T$ of $T_{\rm bin}$ with vertices $v_s$ with digit symbol $s$. 
(ii) An associated connected sum surface $S_\gamma$.} 
\label{f:treeExampleThird_C}
\end{figure}
\end{center}

Consider the map $\gamma$ on the vertices of $T$ with, 
(a) $\gamma(v_\phi)=\bS_0$,
(b) $\gamma(v_s)=\bS_0$ when $\deg(v_s)=3$, 
(c) $\gamma(v_s)=\bS_1$ when $|s|$ is odd and $s=22\dots 2$, (d) $\gamma(v_s)=\bS_0$ for all other vertices. The infinite connected sum surface
$S_\gamma$ is orientable and has infinite genus. It can be shown, using the methods below, that it is homeomorphic to the unit sphere modified by, (i) the deletion of a convergent sequence of distinct points together with its limit,
(ii) the addition of a sequence of diminishing handles that cluster at $z_*$. 
\end{example}



We also note the following.
\medskip

(A) A model surface $S_\gamma$ for a pair $(T,\gamma)$ may be viewed in various ways as the union of  compact bordered surfaces with disjoint interiors.
For example, let $T(v,n,m)$, for $m>n$, be the finite subtree of the tree $T$ with source vertex $v=v_s$ at level $n=|s|$, together with all the  vertices on the paths from $v$ up to level $m$. Then the restriction of a given map $\gamma$ for $T$ to these vertices determines a compact \emph{bordered} surface $S(v,n,m)$ whose boundary is the union of a single ``entrance cycle" and a finite number of disjoint ``exit cycles" according to the degrees of the terminal vertices of $T(v,n,m)$.
Evidently $S_\gamma$ is equal to the union of such bordered surfaces coming from a partition of $T$ into subtrees of the type $T(v,n,m)$. We note that the reduced genus of $S(v,n,m)$  is the sum of the values $g_r(\gamma(w))$ over the vertices of $T(v,n,m)$, where $g_r(\bS_0)=0, g_r(\bP)=1/2, g_r(\bS_1)=1$.  
\medskip

(B) The classification of compact bordered surfaces $S$ ensures that if $S$ has $r+1$ boundary curves, with a distinguished entrance curve and $r$ exit curves, then $S$ is homeomorphic to a surface of the form $S(v,n,m)$ where the homeomorphism matches the distinguished entrance curves. 
\medskip

The model surfaces $S_\gamma$ are connected sums over countable trees that have vertex degrees bounded by 3. However, since concatenations of the branching surface $\bS\backslash 3\bD$ yield general branching surfaces of the form $\bS\backslash r\bD$ it is clear that the surfaces $S_\gamma$ include, up to homeomorphism, similar infinite connected sums over a general countable directed tree. 


\begin{thm}\label{t:modelsufficiency}
Each noncompact surface $S$ is homeomorphic to a surface $S_\gamma$.
\end{thm}

\begin{proof}[Sketch proof] As previously indicated we may assume that $S$ a triangulation $G\subset S$. Since $G$ is a countable union it follows that $S$ is the union of an inclusion chain of connected compact subsets $A_1\subset A_2 \subset \dots ,$ where we assume that each $A_n$ is the closure of the union of a finite set of faces of $G$, that the interior of $A_n$  is connected, and that $A_n$ is contained in the interior of $A_{n+1}$ for all $n$. However, $A_1$ need not be a compact bordered subsurface since a vertex $v$ of $G$ in the boundary of $A_1$ need not have a relative neighbourhood that is homeomorphic to a disc with the image of $v$ a boundary point. 
To remedy this an enlargement $A_1'$ of $A_1$ is constructed in a sufficiently small neighbourhood of the boundary of $A_1$. This may be done explicitly by performing barycentric subdivisions on the faces of $G$ that have closures meeting the boundary of $A_1$, and then defining $A_1'$, and redefining $G$ accordingly, by incorporating new faces in the small neighbourhood. 

The next step is to enlarge $A_1'$ to a similar compact bordered surface $A_1''$ with the additional property that each component of the closure of $S\backslash A_1''$, has boundary consisting of a \emph{single} simple closed curve of $A_1''$ corresponding to a cycle in a further local refinement of $G$. Suppose, for example, that $U$ is a component with just 2 boundary cycles, $c_1, c_2$ of the triangulation of $A_1'$. Then there are faces $F_1, F_2, \dots , F_r$ in the interior of $U$ forming a path between edges $e_1, e_2$, in $c_1, c_2$ respectively, in the sense that  $e_1$ is an edge of $F_1$, $e_2$ is an edge of $F_r$, and each pair $F_i, F_{i+1}$ share a single boundary edge. Augmenting the faces of $A_1'$ by these faces determines a new compact bordered surface with a closed walk replacing $c_1, c_2$. As before, a local barycentric subdivision can be performed to ensure that the closed walk is a cycle of edges.  Repeating this construction gives the desired compact bordered surface $A_1''$, with the components of the complement being in bijective correspondence with the boundary cycles of $A_1''$.

Repeating such constructions yields an inclusion chain of compact bordered surfaces $B_1\subset B_2 \subset \dots$, together with an associated triangulation such that for each $n$  the components of $S\backslash B_n$ have closures that are subsurfaces with a single boundary cycle. Such a chain of bordered surfaces is known as a canonical exhaustion \cite{ahl-sar}.

The inclusion chain determines a general countable directed tree $T=(V,E)$ where, (i) the vertices at level $n$ are labelled by the boundary cycles of $B_n$, (ii) a single source vertex is adjacent to vertices at level 1, (iii) a vertex $v$ at level $n$ is either a terminal vertex or is adjacent to vertices $v'$ at level $n+1$ where the boundary cycle for each $v'$ is contained in the component of the complement of $B_n$ with boundary cycle labelled by $v$. We see then that $S$ is the infinite join  of $B_1$ and the compact bordered subsurfaces that form the connected components of $(B_{n+1}\backslash B_n)^-$, for $n=1,2,\dots$. Such a component bordered subsurface, $\S$ say, has a distinguished entrance boundary cycle and a number, say $r_n\geq 0$, of exit boundary cycles, and at least one of these components has $r_n>0$. Noting (A) and (B) above, each component bordered surface $\S$ is homeomorphic to a topological connected sum of the form
\[
(S_1\backslash \bD)+S_2+S_3+\dots + S_{t-1} + (S_t\backslash r_n\bD), 
\] 
where each $S_i$ is one of $\bS_0, \bS_1$ or $\bP$. Combining these homeomorphisms, by concatenations, gives a homeomorphism between $S$ and a surface associated with an infinite connected sum over a general countable tree and hence, in view of the remark preceding the theorem, to a homeomorphism between $S$ and a model surface $S_\gamma$.
\end{proof}


\begin{cor}\label{c:noterminals}
Each noncompact surface $S$ is homeomorphic to a surface $S_\gamma$ associated with a countable tree $T$ with no terminal vertices.
\end{cor}

\begin{proof}
It suffices to show that any model surface is homeomorphic to a model surface with no terminal vertices. This follows routinely using the principle (B) and joins of homeomorphisms of compact bordered subsurfaces.
\end{proof}

\begin{rem} Theorem \ref{t:modelsufficiency} can also be deduced from Ker\'ekj\'art\'o's theorem by confirming that the model surfaces exhaust all possibilities for the invariants.
This simple confirmation is Theorem \ref{t:richards} below. A quite different set of model surfaces was given in Richards \cite{ric} to obtain this fact. 
\end{rem}

\section{Invariants and Ker\'ekj\'art\'o's theorem} \label{s:invariants}
A noncompact surface $S$ is said to be \emph{orientable} if every compact bordered subsurface is orientable and to be \emph{infinitely nonorientable}  if there is no compact bordered subsurface $A$ for which every component of $S \backslash A$ is orientable.

If $S$ fails to be orientable or infinitely nonorientable, in which case $S$ is said to be \emph{finitely nonorientable}, then it is said to be of  \emph{odd} or \emph{even nonorientability type} according to whether every sufficiently large compact subsurface contains an odd or even number of cross caps. This may also be expressed in terms of the reduced genus, where even (resp. odd) nonorientability type holds if every sufficiently large compact bordered subsurface $A$ is nonorientable and has integral (resp. nonintegral) reduced genus. In our considerations in subsequent sections we do not need to know that the odd/even nonorientability types are distinct properties up to homeomorphism but this is the case.  Indeed, consider the finitely nonorientable surfaces
\begin{equation} 
\label{e:infSums}
S_1=\bP+\sum_\bN \bS_1,\quad  S_2=\bP+\bP+\sum_\bN \bS_1,\\
 \quad S_3=\bP+\bP+\bP+\sum_\bN \bS_1.
\end{equation}
That $S_1$ and $S_3$ are homeomorphic follows readily from the fact that $\bP +\bS_1$ and $\bP +\bP+\bP$ are homeomorphic. On the other hand all sufficiently large  compact bordered subsurfaces $A$ of $S_2$ for which the complement $S_2\backslash A$ has orientable components have the property that $g_r(A)$ is an integer, and so $S_1$ and $S_2$ are not homeomorphic.

The \emph{ideal boundary} of $S$ is defined as a set of equivalence classes of so-called \emph{boundary components} of $S$. Such a component is a decreasing sequence $p=(P_n)$ of pathwise connected noncompact sets,  $p= P_1\supset P_2\supset \dots $, such that, (i) the boundary of $P_i$ in $S$ is a set of disjoint cycles, (ii) for each compact subset $A$ the intersection $A \cap P_n$ is empty for sufficiently large $n$.
The boundary components $p=(P_n)$ and $p'=(P_n')$ are said to be \emph{equivalent} if for any $n$ the inclusion $P_n\subseteq P_N'$ holds for some $N$, and for any $m$ , $P_m'\subseteq P_M$ for some $M$. The {ideal boundary} $\beta(S)$ is then the set of these equivalence classes.
A class $p^*$ is also known as an \emph{end} of the surface. 

The ideal boundary is topologised in the following manner. For any set $U$ in $S$ whose boundary in $S$ is compact the set $U^*$ is defined as the set of all boundary components $p^*$, represented
by a sequence $p = P_1 \supseteq P_2 \supseteq \dots $, such that $P_n\subset U$  for $n$ sufficiently large. The collection of sets $U^*$ is a basis for the topology of $\beta(S)$ and it can be shown that $\beta(S)$ is a nonempty compact, separable and totally disconnected topological space. 
It is routine to show that the equivalence classes $p^*$ of $S_\gamma$ are in bijective correspondence with the maximal infinite directed paths $\pi$ in $T$, and that the ideal boundary $\beta(S_\gamma)$ is homeomorphic to $C_T$. 

A pathwise connected noncompact subset $P$ of $S$ is said to be \emph{planar} if it is homeomorphic to a subset of the plane.

\begin{definition}\label{d:idealboundary}
A point $p^*$ in $\beta(S)$ represented by $p = P_1\supset P_2 \supset \dots $ is planar (resp. orientable) if the sets $P_n$ are planar (resp. orientable) for all sufficiently large $n$.
The subset $\beta'(S)$ (resp. $\beta''(S)$) consists of the points $p^*$ that are planar (resp. orientable). 
\end{definition}

For a model surface $S_\gamma$ with tree $T$ the set of all infinite paths $\pi_x$, for $x$ in $C_T$, that share a particular vertex $v$ determines an open set and these sets provide a base for the topology. It follows that $\beta'(S_\gamma)$ and $\beta''(S_\gamma)$ are open sets. Also the complement of $\beta'(S_\gamma)$ (resp. $\beta''(S_\gamma)$) is the closed subset of points $x$ in $C_T$ whose paths $\pi$ have countably many vertices that are sources of directed paths that include a vertex $w$ with $\gamma(w) \neq \bS_0$ (resp. with $\gamma(w) \neq \bS_0, \bS_1$). 

In the next proof it is convenient to consider the \emph{inflated binary tree} $T_{\rm bin}^*$ which we can take to be the directed subtree obtained from $T_{\rm bin}$ by replacing each edge $vv'$ by edges $vw, wv'$. Once again there is a bijection between the set of infinite paths $\pi$ a Cantor set that we denote as $C_{\rm bin}^*$.

\begin{thm}\label{t:richards}
Let $X$ be a nonempty compact separable totally disconnected Hausdorff space with open subsets $X'\subseteq X''$. Then there is a surface $S_\nu$ and a homeomorphism $\phi: \beta(S_\nu)\to X$ with $\phi(\beta'(S_\nu))=X'$ and $\phi(\beta''(S_\nu))=X''$.
\end{thm}


\begin{proof}
The space $X$ is homeomorphic to a subset of a Cantor set, with the relative topology and so we may assume then that $X$ is a subset of $C_{\rm bin}^*$. Let $T$ be the subtree of $T_{\rm bin}^*$ determined by $X$. We now define a map $\nu:T\to 
\{\bS_0, \bS_1, \bP\}$. Note that for any such map $\beta(S_\nu)=X$. 

Suppose first that $S$ is infinitely nonorientable, that is, that the complement of $X''$ in $X$ is nonempty. Let $\nu(v)=\bS_0$ for $v=v_\phi$ and for each vertex $v$ with degree 3. For $x$ in $Z$ with path $\pi_x$ let $\nu(w)=\bP$ for each vertex $w$ in $\pi_x$ with degree 2. For $x$ in $X''\backslash X'$ with path $\pi_x$ let $\nu(w)=\bS_1$ for each vertex $w$ in $\pi_x$ with degree 2 for which  $\nu(w)$ is not yet defined. Finally, for $x$ in $X'$ with path $\pi_x$ let $\nu(w)=\bS_0$ for each vertex $w$ in $\pi_x$ with degree 2 for which  $\nu(w)$ is undefined. Then 
$\beta'(S_\nu)=X', \beta''(S_\nu)=X''$ and the proof is complete in this case. The remaining cases, with $X''=X$ follow similarly.
\end{proof}

We give a simple proof of Ker\'ekj\'art\'o's theorem based on Theorem \ref{t:modelsufficiency} and the transparency of the invariants for model surfaces.  
The main ingredient is the involvement of ``summand swapping" homeomorphisms. Consider, for example, the infinitely nonorientable noncompact surfaces
\[
S_4=\bS_1+\bP +\bS_1+\bP +\dots , \quad S_\infty=\bP +\bP +\bP +\bP +\dots 
\]
where $\beta(S)=\beta'(S)=\beta''(S)$ is a singleton in both cases. 
Then there is a homeomorphism, 
$\sigma_4:S_4\to S_5$ say, where $S_5=\bP+\bS_1+\bS_1+\bP +\bS_1+ \dots $,
that is the identity map between the complement of the component compact bordered subsurface $\bS_1+\bP\backslash \bD$ in $S_4$ and the complement of $\bP+\bS_1\backslash \bD$ in $S_5$. Similarly there is a summand swapping homeomorphism $\sigma_5: S_5\to S_6$, with 
\[
\sigma_5: S_5=\bP+(\bS_1+\bS_1+\bP) +\bS_1+\dots\quad  \to \quad S_6=
\bP+(\bP+\bS_1+\bS_1) +\bS_1+\dots 
\]
that is the identity map between the complements of the compact bordered surfaces corresponding to the bracketed summands. Continuing, define similarly the maps $\sigma_6, \sigma_7,\dots $. For each point $x$ in $S_4$ the
compositions $\sigma_n\circ \sigma_{n-1}\circ \dots \circ \sigma_4(x)$ are eventually constant as $n$ tends to infinity and so a map $\sigma:S_5 \to S_\infty$ is defined. This map is continuous, injective and surjective, and so a homeomorphism. 

Similarly, a connected sum $\sum_k S_k$ with $S_k=\bP, \bS_1$ or $\bS_0$ is homeomorphic to such a sum in \emph{standard form} in the sense that 
an $\bS_0$-summand cannot precede another type of summand, and an $\bS_1$-summand cannot precede a $\bP$-summand. 

Concatenations of swapping homeomorphisms can be similarly defined between model surfaces, and it similarly follows that every model surface $S_\gamma$ is homeomorphic to one in standard form wherein the same precedences prevail in every infinite path.

\begin{thm}\label{t:Kthm} Let $S_1$ and $S_2$ be noncompact surfaces having the same genus and the same orientability type.
Then $S_1$ is homeomorphic to $S_2$ if and only if there is a homeomorphism $\phi:\beta(S_1)\to \beta(S_2)$ such that $\phi(\beta'(S_1))= \phi(\beta'(S_2))$ and $\phi(\beta''(S_1))= \phi(\beta''(S_2))$. 
\end{thm}

\begin{proof} It remains to show that if  $\phi$ exists then there is a homeomorphism $\sigma:S_1\to S_2.$
By Theorem \ref{t:modelsufficiency} and Corollary \ref{c:noterminals} we may assume that $S_1=S_\gamma$ and $S_2=S_\mu$, where $T_\gamma$ and $T_\mu$ have no terminal vertices. Also we may assume  $T_\gamma \subseteq T_{\rm bin}^*, T_\mu  \subseteq T_{\rm bin}^*$. Since $\phi$ gives a bijection, $\phi_T$ say (with $\phi_T(\pi_x) = \pi_{\phi(x)}$), between the infinite paths of $T_\gamma$ and $T_\mu$, we may assume further that $T_\gamma = T_\mu= T$ and $\phi$ is the identity map on $X=C_T$. The labelling maps $\gamma, \mu$ are now defined on $T$
and we write $X'$ for  $\beta'(S_\gamma)= \beta'(S_\mu)$ and 
$X''$ for  $\beta''(S_\gamma)= \beta''(S_\mu)$.
Each of the surfaces $S_\gamma$ and $S_\mu$ is homeomorphic by a summand swapping homeomorphism to a model surface in standard form. Since $S_\gamma $ and $S_\mu$ have the same genus and orientability type it follows that these standard forms coincide, and so
$S_\gamma $ and $S_\mu$ are homeomorphic.
\end{proof}

\section{Triangulations of noncompact surfaces}\label{s:triangulations of noncompact surfaces}
We show that every noncompact surface $S$ has a $(3,6)$-tight triangulation in the sense of the next definition. This is done in a constructive manner, in the setting of a model surface, by means of moves corresponding to Henneberg 0-extension, vertex-splitting and finite graph substitutions.   
The Maxwell count $f(G)$ of a finite graph $G=(V,E)$ is $3|V|-|E|$.

\begin{definition}\label{d:36tight} 
(i) A finite graph is \emph{$(3,6)$-sparse} if it satisfies the local count 
$f(G')\geq 6$ for each subgraph $G'=(V',E')$ with at least 3 vertices and is \emph{$(3,6)$-tight} if in addition $f(G)=6$. 
(ii) A countable graph $G$ is \emph{$(3,6)$-sparse} (resp. \emph{$(3,6)$-tight}) if there exists a tower $G_1\subset G_2\subset \dots $ of finite $(3,6)$-sparse subgraphs (resp. $(3,6)$-tight subgraphs) whose union is $G$.
\end{definition}

\subsection{Constructive $(3,6)$-tight triangulations}\label{ss:constructive}
The surface $\bS_0\backslash 3\bD$ with boundary $\partial B$ consisting of disjoint cycles $a,b,c$ is depicted in Figure \ref{f:pants2legs}. 
The next lemma will enable the extensions of finite $(3,6)$-tight triangulations that correspond to the branching aspect  of a subtree $T\subset T_{\rm bin}$.
 \begin{center}
\begin{figure}[ht]
\centering
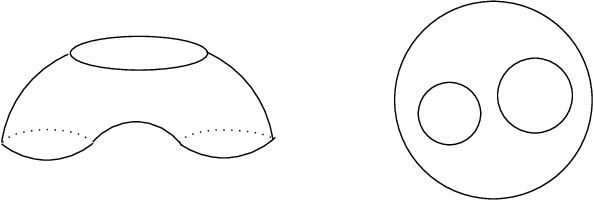
\caption{Depictions of the branching surface $B=\bS_0\backslash 3\bD$.} 
\label{f:pants2legs}
\end{figure}
\end{center}

The Hennenberg 0-extension move, in the context of $(3,6)$-tight graphs, is the graph move $G\to G'$ that adds a single vertex $v$  of degree 3 and  edges $w_1v, w_2v, w_3v$ with $w_1, w_2, w_3$ distinct vertices of $G$. For surface embedded graphs a 0-extension move can be used to introduce a vertex of degree 3 in a face and in this way we may replace any face by new faces. 

\begin{lemma}\label{l:pants2legs}
let $G$ be a $(3,6)$-tight graph with a cycle $d$ of length $|d|\geq 3$. Then for any values of $|b|\geq 3, |c|\geq 3$ with  
$|d|-3=|b|-3+|c|-3$ there is a triangulation $H$ of $B= \bS_0\backslash 3\bD$, with boundary cycles $a, b, c$ of lengths $|a|, |b|, |c|$ with $|a|=|d|$ such that the join $G^+= G\cup_{d=a} H$ is $(3,6)$-tight. Moreover, there exist such triangulations $H$ that may be obtained from the cycle $d$ by 0-extension moves.
\end{lemma}

\begin{proof}
Apply a 0-extension move to the graph of the cycle $d$, viewed as a planar graph, as  illustrated in Figure \ref{f:hennenbergconstn}(i), to create further cycles $b', c'$ of  lengths $|b'|=|b|, |c'|=|c|$. Apply further 0-extension moves, within the face for the $b'$-cycle to obtain a cycle $b$, of the same length, that is disjoint from the cycle $d$, as  illustrated in Figure \ref{f:hennenbergconstn}(ii). Doing the same for $c'$ we obtain a triangulation $H$ of $\bS_0\backslash 3\bD$.
The 0-extension move preserves $(3,6)$-sparsity and $(3,6)$-tightness. Thus the proof of the lemma is completed on observing that $G^+$ is constructed from $G$ by 0-extension moves.
\end{proof}
 \begin{center}
\begin{figure}[ht]
\centering
\includegraphics[width=7cm]{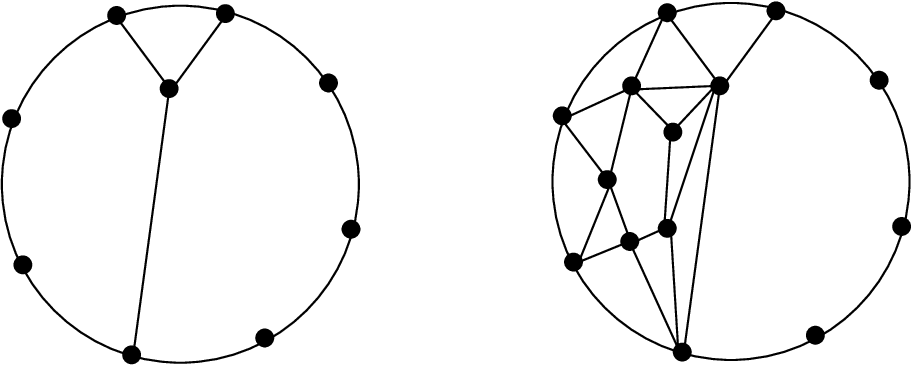}
\caption{(i) A Hennenberg 0-extension move, adding a degree 3 vertex and 3 incident edges, creating cycles of length 5 and 6. (ii) Further 0-extension moves giving a 5-cycle that is disjoint from the 8-cycle and the 6-cycle.} 
\label{f:hennenbergconstn}
\end{figure}
\end{center}

The following elementary substitution lemma features in the proof of Lemma \ref{l:pants1leg} and is useful for sequential constructions. It has a companion lemma, with the $(3,6)$-tight condition replaced by minimal 3-rigidity, that we make use of in Section \ref{s:barjointframeworks}.

\begin{lemma}\label{l:substitution}
Let $A, A'$ be compact bordered surfaces, let $G\subset A, G'\subset A'$ be $(3,6)$-tight triangulations, and let $d, d'$ be cycles of edges of $G, G'$ corresponding to one of the boundary cycles of $A, A'$. Suppose that $B$ is a compact bordered surface with a triangulation $H$ with a boundary cycle $a$ and that $|a|=|d|=|d'|$. Then $G\cup_{d=a} H$ is $(3,6)$-tight if and only if $G'\cup_{d'=a} H$ is $(3,6)$-tight.
\end{lemma}

\begin{proof}
Let $G\cup H$ be $(3,6)$-tight and let  $K\subset G'\cup H$ be a subgraph with at least 3 vertices. We claim that $f(K)\geq 6$. If $K\subset G'$ or $K\subset  H$ this is clear and so we may assume that $K$ has edges in $G'$ and in $H$. Now
\[
6\leq f(G\cup (K\cap H))=f(G)+f(K\cap H) -f((K\cap H)\cap G)
\]
and so $f(K\cap H) -f((K\cap H)\cap G)\geq 0$. Also, $f((K\cap H)\cap G)= f((K\cap H)\cap G')$. Thus
\[
f(K)=f((K\cap G')\cup (K \cap H)=f(K\cap G')+(f(K\cap H) - f(K\cap H\cap G)),
\]
and so $f(K)\geq 6$. It follows similarly that $f(G'\cup H)=6$.
\end{proof}

In the proof of the next lemma we make use of the triangulations given in Figure \ref{f:small36tight}. It is elementary to show that, as graphs, they are generated from $K_3$ by a sequence of 0-extension moves.

We define the \emph{discus graph} $\D_r$, with $r\geq 3$ perimeter vertices, as the $(3,6)$-tight graph obtained from an $r$-cycle by adding two vertices and the $2r$ edges from these vertices to the $r$-cycle vertices. In particular
$\D_3=K_5\backslash e$. In the next key lemma we use 0-extension and vertex-splitting moves on the $(3,6)$-tight embedded graphs of Figure \ref{f:small36tight} to introduce an embedded discus graph adjacent to an enlargement of the nontriangular face, and this discus graph is ultimately substituted by a $(3,6)$-tight triangulation of the bordered surface $G$. 

 \begin{center}
\begin{figure}[ht]
\centering
\includegraphics[width=3cm]{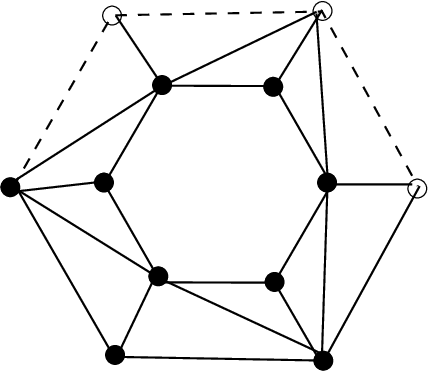}\quad \quad \quad \quad
\includegraphics[width=2.5cm]{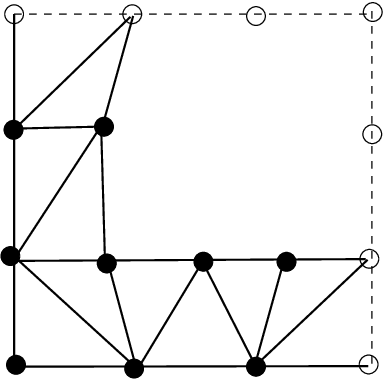}
\caption{$(3,6)$-tight triangulations of the bordered surfaces  (i) $B= \bP\backslash \bD$ with $g_r(B)=1/2$, (ii) 
$B= \bS_1\backslash \bD$ with $g_r(B)=1$. 
} 
\label{f:small36tight}
\end{figure}
\end{center}

 \begin{center}
\begin{figure}[ht]
\centering
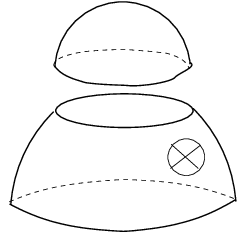
\caption{The case $B=\bP\backslash 2\bD$ for which a triangulation $H$ is required, with $|a|=|d|$ and $|b|=|a|+3$ so that the join of $G$ and $H$ is $(3,6)$-tight.} 
\label{f:projPlaneAddition}
\end{figure}
\end{center}

\begin{lemma}\label{l:pants1leg}
Let $G$ be a $(3,6)$-tight graph with a cycle $d$ of length $\delta\geq3$ and let $B$ be one of the compact bordered surfaces $\bS_0\backslash 2\bD, \bP\backslash 2\bD, \bS_1\backslash 2\bD$.
Suppose that 
$\beta=\delta+6g_r(B)$. Then there is a triangulation $H$ of $B$ with boundary cycles $a, b$ of lengths $\delta, \beta$ such that the join $G^+= G\cup_{d=a} H$ is $(3,6)$-tight 
\end{lemma}


\begin{proof}
Let $H_1$ be the $(3,6)$-tight triangulation of $\bP\backslash \bD$ given in Figure \ref{f:small36tight}(i). Let $K_3\cup_e C_6$ be the subgraph composed of the boundary 6-cycle, $C_6$, of the nontriangular face together with a facial 3-cycle $K_3$ sharing an edge $e$ of this 6-cycle.  
This subgraph is denoted $K_3\cup_e C_6$.
Applying two 0-extension moves to $K_3$ creates $\D_3\cup C_6$, where $\D_3$ is the discus graph with a 3-cycle perimeter.
This move is depicted by the first arrow in Figure \ref{f:BandHmove}.
 \begin{center}
\begin{figure}[ht]
\centering
\includegraphics[width=10cm]{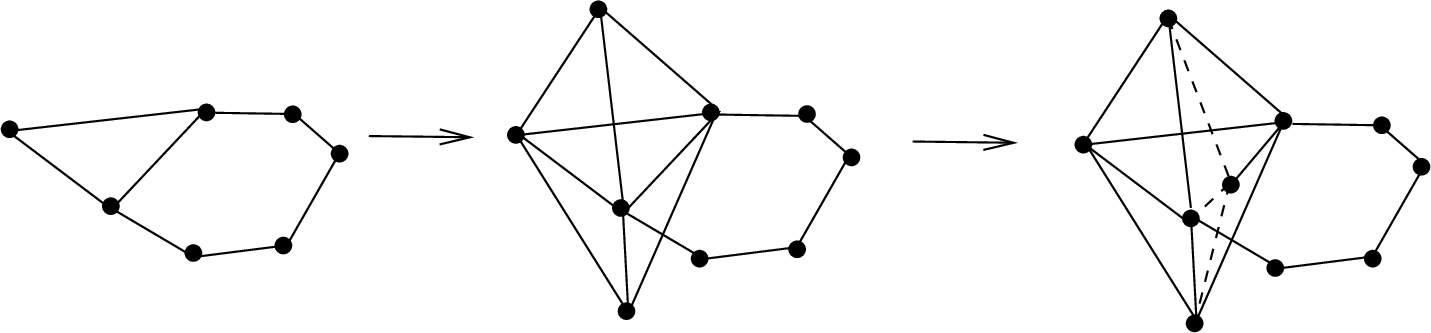}
\caption{A double 0-extension move, $K_3\cup_e C_6 \to \D_3\cup_e C_6$, followed by a vertex-splitting move $\D_3\cup_e C_6 \to \D_4\cup_e C_7$.} 
\label{f:BandHmove}
\end{figure}
\end{center}
The second arrow of this figure depicts a vertex splitting move on a perimeter vertex of  $\D_3$ to create $\D_4\cup C_7$. Such a move, replacing an edge by a vertex and 4 edges in this manner, preserves $(3,6)$-sparsity. Further vertex splitting moves are possible, if necessary, to create $\D_\delta\cup C_{\delta +3}$.

These construction moves on the subgraph of $H_1$ evidently extend to  moves on $H_1$. This results in a graph that contains the discus graph $\D_\delta$, with a perimeter cycle $d'$ of length $\delta$, 
and has the form $\D_\delta \cup_{d'=a} H'$ where $H'$ 
has boundary cycles  $a'$ and $b'$ of lengths $|a'|=\delta$ and $|b'|=\delta+3$.
The cycles $a', b'$ share some vertices, and so $H'$ is not yet a triangulation of $\bP\backslash 2\bD$. However, performing 0-extension moves on the vertices of $b'$ (in the manner of Figure \ref{f:hennenbergconstn}(ii)) yields a $(3,6)$-tight graph of the form $\D_\delta \cup_{d'=a} H$ where $H$ 
is a triangulation of $\bP\backslash 2\bD$ that has disjoint boundary cycles, $a$ and $b$, with lengths $|a|=\delta$ and $|b|=\delta+3$.
The lemma now follows in this case from  Lemma \ref{l:substitution}, by substituting the graph $G$ for the discuss subgraph $\D_\delta$.

The argument for $B=\bS_1\backslash 2\bD$, starting with the triangulation of $\bS_1\backslash \bD$ in Figure \ref{f:small36tight}(ii), is entirely similar, and the argument for $B=\bS_0\backslash 2\bD$ requires only 0-extensions and is a special case  Lemma \ref{l:pants2legs}. 
\end{proof}

\begin{thm}\label{t:S_hastriag}
Every noncompact surface has a triangulation that is $(3,6)$-tight.
\end{thm}

\begin{proof} By Corollary \ref{c:noterminals} we may assume that the surface is a model surface $S_\gamma$ associated with $\gamma: T \to \{\bS_0, \bS_1, \bP\}$ where $T$ is a subtree of $T_{\rm bin}^*$ with no terminal vertices. Note that $S_\gamma$ is equal to the infinite join of a sequence of component bordered surfaces $\S_0, \S_1, \dots, $ of the forms (i) $\bS_0\backslash \bD$ (for $\S_0$ corresponding to the vertex $v_\phi$), (ii)
$\bS_0\backslash 3\bD $ (for $\S_k$ corresponding to vertices of degree 3), and (iii) $S\backslash 2\bD$, where $S$ is 
$\bS_0, \bS_1,$ or  $\bP$ (coresponding to vertices of degree 2). The sequence corresponds to a natural exhaustive enumeration of the vertices of $T$.

Let $G_0$ be any triangulation of $\S_0=\bS_0\backslash \bD$ with boundary cycle a 3-cycle. This is a $(3,6)$-tight triangulation. Suppose, for some $k\geq 0$, that $G_k$ is a $(3,6)$-tight triangulation of the join
$S_k = \S_0\cup \S_1\cup \dots \cup \S_k.$  
If $S_{k+1}$ is obtained from $S_k$ by a join with $\bS_0\backslash r\bD$ for $r=2$ or $3$ then a $(3,6)$-tight triangulation $G_{k+1}$ of $S_{k+1}$ is provided by Lemma \ref{l:pants1leg} or \ref{l:pants2legs}. 
If $S_{k+1}$ is obtained from $S_k$ by a join with $\bS_1\backslash 2\bD$ or $\bP\backslash 2\bD$ then Lemma \ref{l:pants1leg} provides a triangulation $G_{k+1}$ by means of 0-extension moves, vertex splitting moves, and a $(3,6)$-tight subgraph substitution, and so, by  Lemma \ref{l:substitution}, $G_{k+1}$ is $(3,6)$-tight. The union of the triangulation $G_k$ gives the desired triangulation of $S_\gamma$.
 \end{proof}


\begin{rem}\label{r:blockandhole} In Cruickshank, Kitson and Power \cite{cru-kit-pow-1} a complete characterision of $(3,6)$-tight block and hole graphs \cite{fin-whi} with a single block was given in terms of the satisfaction of \emph{girth inequalities}. This gives another approach to Lemma \ref{l:pants1leg} with the $(3,6)$-tight graph $G$ playing the role of a single block. The girth inequalities here correspond to the genus 0 case of higher genus girth inequalities given in the next section.
\end{rem}

\section{General $(3,6)$-tight triangulations.}\label{s:characterisation}
We next give the characterisation of 
$(3,6)$-tight triangulations of 
compact bordered surfaces in terms of certain higher genus girth inequalities. 
This leads to an alternative nonconstructive proof of Theorem \ref{t:S_hastriag} by making use of barycentric subdivisions to ensure the satisfaction of the girth inequalities.

Let $G\subset S$ be a cellularly embedded finite graph in a compact surface $S$. We define a \emph{superface} of $G$ to be a face $U$ of a subgraph $H$ of $G$, where $U$ is not dense and $H$ has no vertices of degree 0 or 1. In particular the complement of $U$ contains at least one face of $G$. A \emph{balanced} superface is one for which the complement of its closure is connected and so also a superface. A 
\emph{simple} superface $U$ is one whose closed boundary walks, denoted $d_1, \dots , d_s$, are disjoint cycles. For a simple superface $U$ the topological boundary of $U$ is the union of the disjoint cycles $d_1, \dots , d_s$, and so the closure $U^-$ of $U$ is a compact bordered surface. 

Let $U$ be a  superface of $S$ which is both simple and balanced, with complementary superface $W$, so that both $U^-$ and $W^-$ are connected bordered surfaces. Then the following standard addition formula holds for the reduced genus. See also Richards \cite{ric}.
\[
g_r(S)=g_r(U^-)+g_r(W^-)+(s-1).
\]

If $G\subset S$ is cellular with  $f(G)=3v-e \geq 6$ and with nontriangular faces, $U_1, \dots , U_n$, with closed boundary walks $c_1, \dots,  c_n$, then the Euler formula for $G$ gives
\[
\sum_k (|c_k|-3) = 6g_r(S) + f(G)-6.
\]
We refer to this equality as the \emph{face walk identity} for the embedded graph $G$ in $S$. 

Let $G_W$ denote the subgraph of $G$ whose edges lie in the closure of $W$, the complementary superface of $U$, as above.
This embedded graph in $S$ is also an embedded graph in the bordered surface $W^-$. It is also an embedded graph in the surface, $S_W$ say, obtained by capping the $s$ boundary cycles of $U$ with open discs. Since $g_r(S_W)=g_r(W^-)= g_r(W)$ we may use these quantities interchangeably. It follows that we have a face walk identity for $G_W$ in $S_W$, namely 
\[
\sum_{k=1}^s (|d_k|-3) + \sum_{k\in I(W)} (|c_k|-3) = 6g_r(S_W) + f(G_W)-6.
\]
Here we write $I(W)$ for the set of indices $k$ for which
the face $U_k$ of $G$ is contained in $W$. 

\begin{lemma}\label{l:GIforSymmetric_UKW_CYCLES_B}
Let $G$ be a cellularly embedded graph in the compact surface $S$ with $f(G)=6$, let $U$ be a balanced, simple superface of $G$ with complementary superface $W$, and let $d_1, \dots , d_{s}$ be the common boundary cycles of  $U$ and $W$. Then the following inequalities are equivalent.
\medskip

(i) \quad  $f(G_W) \geq 6$.

(ii) 
\[
\sum_{k=1}^s (|d_k|-3) + \sum_{k\in I(W)} (|c_k|-3) \geq 6g_r(W).
\]

(iii) 
\[
\sum_{k=1}^s (|d_k|-3) \geq \sum_{k\in I(U)} (|c_k|-3)-6 (g_r(U)+s-1).
\]

\end{lemma}

\begin{proof}
The equivalence of (i) and (ii) follows from the face walk identity for $G_W$ in $S_W$. By the face walk identity for $G$ observe that (ii) holds if and only if
\[
\sum_{k=1}^s (|d_k|-3) + (6g_r(S)-\sum_{k\in I(U)} (|c_k|-3)) \geq 6g_r(W),
\]
and, in view of the reduced genus addition formula, this holds if and only if (iii) holds.
\end{proof}

In the next definition we refer to general balanced superfaces $U$. Thus  $U$ has a set of closed boundary walks, denoted once again as $d_1, \dots , d_s$, and these walks need not be cycles and need not be disjoint. Also, as simple examples show, the boundary of $W$ need not be equal to the boundary of $U$.

\begin{definition}\label{d:girthinequalities}
A cellularly embedded graph $G$, in the compact surface $S$, with $f(G)=6$,  \emph{satisfies the girth inequalities}
if the inequality (iii) of Lemma \ref{l:GIforSymmetric_UKW_CYCLES_B} holds for every balanced superface $U$ of $G$ with closed boundary walks $d_1, \dots ,d_s$. 
\end{definition}


\begin{thm}\label{t:36tightinG}
Let $G$ be a cellularly embedded graph in a compact surface with $f(G)=6$. Then the following are equivalent.

(i) $G$ is $(3,6)$-tight.

(ii) $f(G_U)\geq 6$ for every superface $U$ of $G$. 

(iii) $G$ satisfies the girth inequalities.
\end{thm}

\begin{proof}
That (i) implies (ii) is immediate and the converse is due to Qays Shakir \cite{sha}. We give a proof using superfaces \cite{pow-girth}.
Suppose then that (ii) holds and (i) is not true. Then there is a maximal subgraph $K$ of $G$, with at least 3 vertices and  $f(K)\leq 5$. 
Let $U$ be a face of $K$ and suppose first that the boundary walks of $U$ consist of disjoint cycles. Then the complement of the closure of $U$ is the union of components, say $W_1, \dots ,W_r$, that are superfaces of $G$ and have disjoint closures. 
Let $G_U^c$ be the subgraph of $G$ induced by the edges of $G$ that lie in the complement of $U$. In view of our assumption $G_U^c$ is the union of the  disjoint subgraphs,
$G_{W_1}, \dots , G_{W_r}$ and so by (ii) $f(G_U^c)\geq 6r$. 
Also,  $f(G)=f(G_U) + f(G_U^c) - f(G_U \cap G_U^c)$ and so 
$f(G_U) - f(G_U \cap G_U^c)\leq 0$.
Our assumptions imply $\partial U = K \cap G_U=G_U \cap G_U^c$. We have 
$f(K\cup G_U)=f(K) + (f(G_U) - f(K \cap G_U)),$
and so it follows that $f(K\cup G_U)\leq 5$. Thus $K=K\cup G_U$ by the maximality of $K$. Since this is true for all faces $U$ of $K$ it follows that $K=G$, a contradiction, as required. 

In the general case perform vertex splitting moves on all the vertices of the boundary walks of $U$, adding the new vertices to the interior of $U$. The result of these moves, denoted $G \to G^+$, define an embedded graph $G^+\subset S$ with superface $U^-\subset U$, whose boundary walks are disjoint cycles. The boundary walks of $G_U$ belong to $K$ and the associated move $K \to K^+$ provides a subgraph $K_+$ with $K \subset K^+\subset G^+$ and $f(K^+)=f(K)\leq 5$. As in the previous paragraph, since $K^+\cap G_{U_-}= G_{U_-}^c \cap G_{U_-}$ and $f(G^+)=6$ it follows that
$f(K^+\cup G^+_{U_-})\leq5$. Since $K\cup G_U \to K^+\cup G^+_{U_-}$ it follows that
$f(K\cup G_U) = f(K^+\cup G^+_{U_-})\leq 5$ and so by the maximality of $K$ we have $K=K\cup G_U$, leading to a contradiction.

The equivalence of (ii) and (iii) follows in two routine steps. (See also \cite{pow-girth}.) Firstly, Lemma \ref{l:GIforSymmetric_UKW_CYCLES_B} can be generalised to balanced superfaces $U$ that are not necessarily simple. This follows from the association of a balanced superface of $G$ with a balanced simple superface of a larger embedded graph in $S$ obtained by vertex splitting moves. Secondly (ii) is equivalent to having $f(G_U)\geq 6$ for every balanced superface.
\end{proof}

We next consider a local barycentric subdivision move $H \to H'$ for a triangulation $H$ of a compact bordered surface $S$. Such a move, defined on an edge $e$, is illustrated in Figure \ref{f:barycentric}.
Note that $f(H')=f(H)$, since $e$ has been replaced by 10 edges and 3 vertices.

\begin{center}
\begin{figure}[ht]
\centering
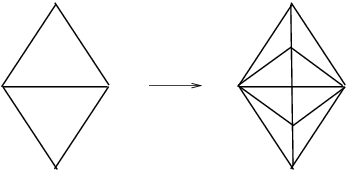 
\caption{A local barycentric move for the edge $e$.} 
\label{f:barycentric}
\end{figure}
\end{center}

\begin{lemma}\label{l:barycentricLemma}
Let $S$ be a compact bordered surface, with boundary  curves $c_1, \dots , c_r$, and let $G$ be a triangulation of $S$ with $f(G)=6$. Then for some $n$ the result $G'$ of $n$-fold barycentric subdivision is $(3,6)$-tight.
\end{lemma}

\begin{proof}
Let $\tilde{S}$ be the compact surface obtained by capping each of the border curves by a closed disc. We may then view $G$ as a cellularly embedded graph in $\tilde{S}$ with $f(G)=6$ and where the boundary cycle of a nontriangular face is $c_i$ for some $i$. Discarding the cycles $c_i$ of length 3 and relabelling, we may assume that $G\subset \tilde{S}$ has $r$ nontriangular faces $U_1, \dots , U_r$ with boundary cycles $c_1, \dots ,c_r$.


By Theorem \ref{t:36tightinG} it suffices to show that if $n$ is  sufficiently large then for every balanced superface $U$ of $G'$ we have $\delta(U)\geq 0$ where $\delta(U)$ is the difference 
\[\delta(U)= \sum_{k=1}^s (|d_k|-3)- (6g_r(U)- \sum_{k\in I(U)} (|c_k|-3)).
\]
Let $\B(G)$ be the set of such superfaces for which this inequality fails to hold and let $\B'(G)$ be the subset for which the left hand side has the minimum value, $\mu$ say. If $\mu\geq 0$ then $G$  already satisfies the girth inequalities and there is nothing to prove. Suppose then that $\mu<0$  and $U$ belong $\B'(G)$. We show that there exists a local barycentric move, $G \to G_1$ say, such that
$|\B'(G_1)|<\B'(G)|$.  Repeating such moves sufficiently often leads to $G'$ with $\delta(U)\geq 0$ for all superfaces, as desired.

Since $U$ is a superface of $G$, with proper closure, there exists an edge $e$ in one of the boundary walks, $d_i$ say,  that is not a boundary edge of $G$ and so is incident to 2 triangular faces. Let $G \to G'$ be the local barycentric subdivision move for $e$. Let $U_1$ be a superface of $G_1$ with a boundary walk $d$ that has a subwalk $\pi$ with edges amongst the 10 replacement edges for $e$. Then there is an associated superface $W$ of $G$ with such subwalks replaced by the edge $uv$ where $u, v$ are the initial and final vertices of $\pi$. Since each $\pi$ has length at least 2 it follows that $\delta(U_1)>\delta(W)\geq \mu$, and so $|\B'(G_1)|<|\B'(G)|$.
\end{proof}

\begin{lemma}\label{l:extension}
Let $S_1, T$ be compact bordered surfaces and let $S_2=S_1\cup_{d=a} T$ be their join over  boundary curves $d, a$ in $S, T$ respectively. If $G_1\subset S_1$ is a $(3,6)$-tight triangulation then there is a triangulation $H$ of $T$ such that the boundary cycle $d'$ for $d$ has the same length as the cycle $a'$ for $a$ and the join $G_1\cup_{d'=a'}H$ is $(3,6)$-tight.
\end{lemma}

\begin{proof}
An initial triangulation $H_0\subset T$ is chosen so that $f(G_1\cup_{d'=a'}H_0)=6$. A similar argument to the previous proof shows that after sufficiently many local barycentric subdivision moves $H_0\to H_1 \to \dots\to  H_n=H$,  the join $G_1\cup_{d'=a'}H$ satisfies the girth inequalities. 
\end{proof}

An alternative proof of Theorem \ref{t:S_hastriag} now follows from the previous lemma and the existence of a canonical exhaustion of a noncompact surface $S$ by compact bordered surfaces, as given in Theorem \ref{t:modelsufficiency}.

\begin{rem}
Lemma \ref{l:GIforSymmetric_UKW_CYCLES_B}(ii) can be used to show that a periodic surface in $\bR^3$ with infinite genus fails to have a periodic $(3,6)$-tight triangulation. For example, the bordered surface in Figure \ref{f:periodicsurface} shows a geometric realisation, $S_0$ say, of $\bS_0\backslash 6\bD$, with the symmetry of a cube. We assume that the distance between opposite boundary circles is unity. The 3-periodic surface, $S$ say, formed by an infinite concatenation of translates of $S_0$ has infinite genus and a singleton ideal boundary. The minimal surface realisation of $S$ is known as the Schwarz P surface. 
\begin{center}
\begin{figure}[ht]
\centering
\includegraphics[width=3.5cm]{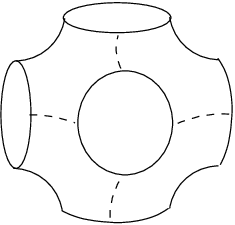}
\caption{Building block for the Schwarz P surface.} 
\label{f:periodicsurface}
\end{figure}
\end{center}

Let $S_n$ be the compact bordered subsurface formed by the join of  $n^3$ copies of $S_0$, in cubic form, and suppose that $S_n$ lies in the interior of a similar subsurface $S_m$ with $m >n$. Suppose, for simplicity, that $G$ is a periodic triangulation of $S$, under integral translations, determined by a triangulation of $S_0$. Let $W$ be the balanced simple superface of $G_m\subset S_m$ that is equal to the interior of $S_m\backslash S_n$. 
It has a set of boundary cycles, $d_1, \dots d_t$ say, consisting of the union of the boundary cycles for $S_n$ and $S_m$. From Lemma \ref{l:GIforSymmetric_UKW_CYCLES_B}(ii) we see that the condition $f(G_W)\geq 6$ is equivalent to
\[
\sum_{k=1}^t (|d_k|-3) \geq 6g_r(W).
\]
By the periodicity, for fixed $n$ the left hand side is of order $m^2$. However, the right hand side is of order $m^3$ and so the inequality $f(G_W)\geq 6$ cannot hold for large $m$ and so $G$ fails to be $(3,6)$-tight.
\end{rem}

\section{Bar-joint frameworks}\label{s:barjointframeworks}
A bar-joint framework $(G,p)$ in $\bR^3$ consists of a finite or countable simple graph $G = (V, E)$ and a placement $p : V \to \bR^3$ such that $p(v) , p(w)$ are distinct for each edge $vw$ in $E$. An \emph{infinitesimal flex} of $(G, p)$ is a velocity assignment $u : V \to \bR^3$ that satisfies the infinitesimal flex condition  $(u(v) - u(w)) \cdot (p(v) - p(w)) = 0$ for every edge $vw$. A \emph{trivial infinitesimal flex} of $(G, p)$ is one that  extends to an infinitesimal flex of any containing framework, which is to say that it is a linear combination of a translation infinitesimal flex and a rotation infinitesimal flex. The framework $(G, p)$ is \emph{infinitesimally rigid} if
the only infinitesimal flexes are trivial, and the graph $G$ is \emph{3-rigid}  if every generic framework $(G, p)$, where the coordinates of the vertex placements form an algebraically independent set, 
is infinitesimally rigid. Also $G$ is \emph{minimally 3-rigid} if on  deleting any edge of $G$ the resulting graph fails to be 3-rigid.
A finite graph $G$ that is minimally 3-rigid is necessarily $(3,6)$-tight, but the converse does not hold as evidenced by the double banana graph. See also Figure \ref{f:DBtorus} below.

It is well-known that the Henneberg 0-extension move and vertex-splitting preserve minimal 3-rigidity as well as $(3,6)$-tightness. See, for example, Whiteley \cite{whi}, Graver et al \cite{gra-ser-ser}, and the appendix of \cite{cru-kit-pow-1}.
A sufficient condition for minimal 3-rigidity is that there exists a sequence of finite subgraphs $G_1\subset G_2 \subset \dots ,$  with union $G$, each of which is minimally 3-rigid. That this sequential 3-rigidity condition is not necessary is shown by examples in Kitson and Power \cite{kit-pow-I}, \cite {kit-pow-II}.

Let us first note some countable surface graphs corresponding to triangulations of low genus noncompact surfaces where the ideal boundary is finite. 

\begin{example}\label{e:torus2holesflex}
Figure \ref{f:DBtorus} (see also \cite{cru-kit-pow-2}) shows an embedded graph in the torus with $12$ vertices that is $(3,6)$-tight and 2-connected. The 2 nontriangular faces each have a single boundary walk of length 6 that is not a cycle, and $G$ can be viewed also as a triangulation of a pseudosurface. Since $G$ is 2-connected it is not 3-rigid.  
\begin{center}
\begin{figure}[ht]
\centering
\includegraphics[width=3cm]{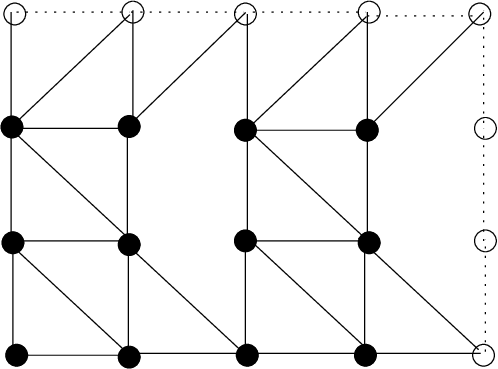}
\caption{A finite $(3,6)$-tight embedded graph $G\subset \bS_1$ that is not 3-rigid.} 
\label{f:DBtorus}
\end{figure}
\end{center}
We may extend this example to an infinite triangulation of the the noncompact surface $S=\bS_1\backslash \{x_1, x_2\}$ by indefinitely  repeated 0-extension moves on the vertices of the 2 boundary walks, in the manner of Figure \ref{f:hennenbergconstn}(ii). This results in a $(3,6)$-tight triangulation of $S$ that is not 3-rigid. 
\end{example}

\begin{example}\label{e:puncturedP}
In Kastis and Power \cite{kas-pow} it is shown that every $(3,6)$-tight embedded graph $G\subset \bP$ is minimally 3-rigid and moreover is obtained from $K_3$ by general vertex splitting moves. (See also \cite{pow-girth}.) It follows from this that any $(3,6)$-tight triangulation of a finitely punctured projective plane is minimally 3-rigid.
\end{example}

The following rigid subgraph substitution lemma (a companion to Lemma \ref{l:substitution}) is well-known and follows from the definitions of infinitesimal flexibility and 3-rigidity.

\begin{lemma}\label{l:3rigidsubstitution}
Let $A, A'$ be compact bordered surfaces, let $G\subset A, G'\subset A'$ be minimally 3-rigid triangulations, and let $d, d'$ be cycles of edges of $G, G'$ corresponding to one of the boundary cycles of $A, A'$. Suppose that $B$ is a compact bordered surface with a triangulation $H$ with a boundary cycle $a$ and that $|a|=|d|=|d'|$. Then $G\cup_{d=a} H$ is minimally 3-rigid  if and only if $G'\cup_{d'=a} H$ is minimally 3-rigid.
\end{lemma}

\begin{thm}\label{t:min3rigid}
Let $S$ be a noncompact surface. Then there is a $(3,6)$-tight triangulation of $S$ whose graph is minimally 3-rigid.
\end{thm}

\begin{proof}
In the proof of Theorem \ref{t:S_hastriag} the $(3,6)$-tight triangulation $G$ of a model surface $S$ is constructed as a union of graphs $G_1\subset G_2\subset \dots ,$
where each $G_k$ is a $(3,6)$-tight triangulation of a compact bordered subsurface $S_k$. The first embedded graph  $G_0$ is a triangulation of $S_0\backslash \bD$ 
with a single triangular boundary and this is minimally 3-rigid as well as $(3,6)$-tight. 
Also, the construction, for any $k\geq 0$, of $G_{k+1}$ from $G_k$ is by means of $0$-extension moves, vertex splitting moves and the substitution of a discus graph by $G_k$. A discus graph is a triangulated sphere and is minimally 3-rigid and so by Lemma \ref{l:3rigidsubstitution} all the construction moves preserve minimal 3-rigidity. Thus  the triangulation of $S_\gamma$ given by the union of the graphs $G_k$ is minimally 3-rigid.
\end{proof}

\subsection{Further directions}\label{ss:openprobs} (a) In contrast to Example \ref{e:torus2holesflex} the $(3,6)$-tight triangulations of the singly punctured torus are minimally 3-rigid. This follows from the main result in \cite{cru-kit-pow-2} that the graph of a $(3,6)$-tight triangulation of $\bS_1\backslash \bD$ is minimally 3-rigid.
This suggests the following interesting problem.
For which compact surfaces $S$ are the $(3,6)$-tight triangulations of $S\backslash \bD$ minimally 3-rigid?

(b)  The rigidity of bar-joint frameworks in $\bR^3$ with respect to nonEuclidean norms, such as the $\ell^p$ norms $\|\cdot\|_p$ for $1< p<\infty, p\neq 2$, is a topic of current interest. See, for example, \cite{kit-pow}, \cite{kit-nix-sch}, \cite{dew-kit-nix}. Because of the absence of infinitesimal rotations it is the $(3,3)$-tightness of the underlying graph $G$ that is a necessary condition for the minimal 3-rigidity of a finite framework with respect to these norms. It is conjectured that it is also a necessary condition. The construction methods of the previous sections for triangulations of noncompact surfaces $S$ are expected to adapt readily to the the existence of $(3,3)$-tight triangulation when $S$ has nonzero genus. Thus the following general problem arises. For which compact surfaces $S$ and integers $r\geq 1$ are the $(3,3)$-tight triangulations of $S\backslash \bD$ minimally 3-rigid for the $\ell^p$ norms?





\bibliographystyle{abbrv}
\def\lfhook#1{\setbox0=\hbox{#1}{\ooalign{\hidewidth
  \lower1.5ex\hbox{'}\hidewidth\crcr\unhbox0}}}

\end{document}